\documentclass{amsart}
\usepackage{amsmath}
\usepackage{amsfonts}

\newtheorem{theorem}{Theorem}[section]
\theoremstyle{plain}

\newtheorem{definition}{Definition}[section]
\newtheorem{example}{Example}[section]

\newtheorem{lemma}{Lemma}[section]

\numberwithin{equation}{section}
\input{tcilatex}

\begin{document}
\title[Topological Nearly Entropy]{Topological Nearly Entropy on Nearly
Compact Spaces}
\author{Zabidin Salleh}
\address{Department of Mathematics, Faculty of Ocean Engineering Technology and Informatics, Universiti Malaysia
Terengganu, 21030 Kuala Nerus, Terengganu, Malaysia.}
\email{zabidin@umt.edu.my}
\author{Syazwani Gulamsarwar}
\curraddr{School of Informatics and Applied Mathematics, Universiti Malaysia
Terengganu, 21030 Kuala Nerus, Terengganu, Malaysia.}
\email{syzwani.g@gmail.com}
\date{August 6, 2019}

\begin{abstract}
In our previous paper \cite{gulam}, we have introduced topological nearly
entropy, $Ent_{N}\left( f\right) $ by restricting $X$ into a class of nearly
compact spaces. In the present paper, some additional properties of this
notion are studied. Furthermore, we introduce another new notion of
topological nearly entropy of $f$ denoted by $Ent_{n}\left( f\right) $ when
the whole space $X$ itself is nearly compact. We show the relationship
between these two notions for the class of nearly compact subspaces. We also
propose new space, namely, $R$-space in studying the topological nearly
entropy on nearly compact and Hausdorff space. As a consequence, the
topological nearly entropy of $f$ and it restriction $f|_{K}$ coincides.
Finally, some fundamental properties of topological nearly entropy for
product space are obtained.
\end{abstract}

\maketitle

\section{Introduction}

Compactness is a crucial property of topological space which play an
important role for the development of topological entropy of a dynamical
system. Adler et al. \cite{adler} introduced the concept of topological
entropy of continuous mapping for a compact space of dynamical system in
1965. Subsequently, Bowen \cite{bowen} generalized the concept of
topological entropy due to Adler et al. \cite{adler} which is
metric-dependent and hold for noncompact spaces, see also \cite{zabidin} for
the detail explanation of these two notions. Next, Liu et al. \cite{liu}
generalized the notion of Adler et al. where separation axioms not
necessarily required which intiate the concept of topological entropy for
arbitrary topological spaces. In 2015, Afsan \cite{uzzal} defined
topological $H$-entropy for topological $H$-dynamical system based on $H$%
-set and $H$-closed topological space.

In a recent paper, Gulamsarwar and Salleh \cite{gulam} presented a new
generalized notion of topological entropy on nearly compact space by using $%
R $-map \cite{carnahan}, namely topological nearly entropy. Recall that
every compact space is nearly compact but not the converse. Thus, the
concept of topological nearly entropy for topological $R$-dynamical system
was initiated in accordance with several fundamental properties of Adler et
al. \cite{adler} and Liu et al. \cite{liu}. Until now, many researchers are
interested with the concept of topological entropy as we can see more of
their works in \cite{canovas} and \cite{goodwyn}, etc.

This paper explore more fundamental properties of topological nearly
entropy. Moreover this paper presenting the notion of topological nearly
entropy for $R$-map on nearly compact space and some its fundamental
properties are studied. These fundamental properties are further extension
of Gulamsarwar and Salleh \cite{gulam} topological nearly entropy.
Furthermore, we introduce $R$-space in order to study the effect of nearly
compactness on Hausdorff space to the topological nearly entropy. Finally,\
we show that $R$-map preserves nearly compact property and obtain several
properties of topological nearly entropy for product space.

In Section 3, we explore more properties and show the results obtained for
topological nearly entropy of the funtion $f$ and it inverse $f^{-1}$ if $f$
is bijective. The definition of topological nearly entropy on nearly compact
space and its fundamental properties are discussed in Section 4. We also
study topological nearly entropy on nearly compact and Hausdorff space which
has triggered the idea to introduce $R$-space. Consequently, we obtain
several properties of topological nearly entropy for product space which are
discusses in the Section 5.

\section{Preliminaries}

Throughout this paper $X$ is always mean arbitrary topological space unless
explicitly stated. We denote the closure and interior of $A\subseteq X$ by $%
\limfunc{cl}\left( A\right) $ and $\limfunc{int}\left( A\right) $,
respectively. A subset $A$ of $X$ is said to be regular open if and only if $%
A=\limfunc{int}\left( \limfunc{cl}\left( A\right) \right) $. By regular open
cover of $X$, we mean a cover of $X$ by regular open sets of $X$. The join
of two covers $\mathcal{U}$ and $\mathcal{V}$ are described as $\mathcal{U}%
\vee \mathcal{V}=\left\{ A\cap B:A\in \mathcal{U},B\in \mathcal{V}\right\} $
and a cover $\mathcal{V}$ is said to be refinement of a cover $\mathcal{U}$,
denoted as $\mathcal{U}\prec \mathcal{V}$, if for every element $V\in
\mathcal{V}$ there exists an element $U\in \mathcal{U}$ such that $%
V\subseteq U$. The complement of a set $A$ is denoted as $A^{c}$.

A subset $Y$ of $X$ is said to be nearly compact relative to $X$ if every
cover $\mathcal{U}=\{U_{\alpha }:\alpha \in \Delta \}$ of $X$ by open sets
of $X$ contains a finite subfamily $\mathcal{U}_{0}=\{U_{1},U_{2},\ldots
,U_{n}\}$ such that $Y\subseteq \bigcup_{i=0}^{n}\limfunc{int}\left(
\limfunc{cl}\left( U_{i}\right) \right) $, or equivalently, $Y$ is covered
by finitely regular open sets of $X$. If $Y=X$ is nearly compact relative to
$X$, the space $X$ is called nearly compact. In other words, the space $X$
is nearly compact if and only if every regular open cover of $X$ contains a
finite subcover, see \cite{singal}.

A function $f:X\rightarrow Y$ is called $R$-map \cite{carnahan} if the
inverse image of every regular open set in $Y$ is regular open in $X$. If $%
f:X\rightarrow X$ is an $R$-map, then the pair $\left( X,f\right) $ is
called topological $R$-dynamical system, and if $X$ is nearly compact, $%
\left( X,f\right) $ is called nearly compact $R$-dynamical system, see \cite%
{gulam}. Now, we recall some definitions of topological nearly entropy and
its fundamental properties from \cite{gulam} as follows.

\begin{definition}
\cite{gulam} Let $\left( X,f\right) $ be a topological $R$-dynamical system,
$\mathcal{U}$ be a regular open cover of $X$ and $K$ be a nonempty nearly
compact relative to the space $X$ such that $f\left( K\right) \subseteq K$.
Let $N_{K}\left( \mathcal{U}\right) =\min \left\{ \limfunc{card}\left(
\mathcal{V}\right) :\mathcal{V}\text{ is a subcover of }\mathcal{U}%
,K\subseteq \dbigcup\limits_{V\in \mathcal{V}}V\right\} $. Since $K$ is a
nearly compact relative to $X$, $N_{K}\left( \mathcal{U}\right) $ is a
positive integer. Let $M_{K}\left( \mathcal{U}\right) =\log N_{K}\left(
\mathcal{U}\right) $.
\end{definition}

\begin{theorem}
\label{Th2.1}\cite{gulam} Let $\left( X,f\right) $ be a topological $R$%
-dynamical system. Let $\mathcal{U}$ and $\mathcal{V}$ be a regular open
covers of $X$, and $K$ be a nearly compact relative to $X$ such that $%
f\left( K\right) \subseteq K$. Then the following statements hold:

$\left( a\right) $ $M_{K}\left( \mathcal{U}\right) \geq 0$

$\left( b\right) $ $\mathcal{U}\prec \mathcal{V}$ implies $M_{K}\left(
\mathcal{U}\right) \leq M_{K}\left( \mathcal{V}\right) $

$\left( c\right) $ $M_{K}\left( \mathcal{U}\vee \mathcal{V}\right) \leq
M_{K}\left( \mathcal{U}\right) +M_{K}\left( \mathcal{V}\right) $

$\left( d\right) $ $M_{K}\left( f^{-1}\left( \mathcal{U}\right) \right) \leq
M_{K}\left( \mathcal{U}\right) $. When $f\left( K\right) =K$, the equality
holds.
\end{theorem}

Let $f:X\rightarrow X$ be an $R$-map and let $\left( X,f\right) $ be a
topological $R$-dynamical system. Denote by $H\left( X,f\right) $ the set of
all $f$-invariant nonempty nearly compact relative to $X$, i.e.,
\begin{equation*}
\left\{ K\in P\left( X\right) \setminus \left\{ \emptyset \right\} :\text{ }K%
\text{ is nearly compact relative to }X\text{ such that }f(K)\subseteq
K\right\} .
\end{equation*}%
If $X$ is nearly compact, it follows from $f\left( X\right) \subseteq X$
that $H\left( X,f\right) \neq \emptyset $. However, when $X$ is not nearly
compact, $H\left( X,f\right) $ could be empty.

\begin{definition}
\cite{gulam}\label{def_N1} Let $\left( X,f\right) $ be a topological $R$%
-dynamical system. $H\left( X,f\right) $ be the set $\left\{ K\in P\left(
X\right) \setminus \left\{ \emptyset \right\} :K\text{ is nearly compact
relative to }X\text{ such that }f\left( K\right) \subseteq K\right\} $ and $%
\mathcal{U}$ be a regular open cover of $X$. Then for $K\in H\left(
X,f\right) $
\begin{equation*}
Ent_{N}\left( f,\mathcal{U},K\right) =\lim_{n\rightarrow \infty }\frac{1}{n}%
M_{K}\left( \bigvee_{i=0}^{n-1}f^{-i}\left( \mathcal{U}\right) \right)
\end{equation*}%
is called the topological nearly entropy of $f$ on $K$ relative to $\mathcal{%
U}$ and
\begin{equation*}
Ent_{N}\left( f,K\right) =\sup_{\mathcal{U}}\left\{ Ent_{N}\left( f,\mathcal{%
U},K\right) :\mathcal{U}\ \text{is a regular open cover of}\ X\right\}
\end{equation*}%
is called topological nearly entropy of $f$ on $K$.
\end{definition}

\begin{theorem}
\cite{gulam}\label{Th 2.2} Let $\left( X,f\right) $ be a topological $R$%
-dynamical system, $K_{1},K_{2}\in H\left( X,f\right) $ with $K_{1}\subseteq
K_{2}$ and $\mathcal{U}$ be a regular open cover of $X$. Then

$\left( a\right) $ $Ent_{N}\left( f,\mathcal{U},K_{1}\right) \leq
Ent_{N}\left( f,\mathcal{U},K_{2}\right) $ and

$\left( b\right) $ $Ent_{N}\left( f,K_{1}\right) \leq Ent_{N}\left(
f,K_{2}\right) $.
\end{theorem}

\begin{definition}
\cite{gulam}\label{def_N2} Let $\left( X,f\right) $ be a topological $R$%
-dynamical system. Then%
\begin{equation*}
Ent_{N}\left( f\right) =\sup_{K}\left\{ Ent_{N}\left( f,K\right) :K\in
H\left( X,f\right) \right\}
\end{equation*}%
is called topological nearly entropy of $f$. For the special case $H\left(
X,f\right) =\emptyset ,Ent_{N}\left( f\right) =0$ is to be taken.
\end{definition}

\section{Further Properties of Topological Nearly Entropy}

Several fundamental properties of topological nearly entropy has been proved
in our previous paper (see \cite{gulam}). Now we shall prove another further
result about topological nearly entropy. The topological nearly entropy of a
map $f$ is coincide with topological nearly entropy of it inverse $f^{-1}$
whenever $f,f^{-1}$ are bijections and $R$-maps as follows.

\begin{theorem}
Let $\left( X,f\right) $ be a topological $R$-dynamical system and the
function $f:X\rightarrow X$ is bijective. Assume that $f^{-1}:X\rightarrow X$
is $R$-map, then if $K\in H\left( X,f\right) $ and $f,f^{-1}:K\rightarrow K$
are bijections and $R$-maps, then $Ent_{N}\left( f\right) =Ent_{N}\left(
f^{-1}\right) $.
\end{theorem}

\begin{proof}
If $H\left( X,f\right) =\emptyset $, then $Ent_{N}\left( f\right) =0$ and $%
Ent_{N}\left( f\right) \leq Ent_{N}\left( f^{-1}\right) $. So, let $H\left(
X,f\right) \neq \emptyset $. Now consider $K\in H\left( X,f\right) $, from
the asumption $f,f^{-1}:K\rightarrow K$ are bijections and $R$-maps, we have
$K\in H\left( X,f^{-1}\right) $ so that $H\left( X,f\right) \subseteq
H\left( X,f^{-1}\right) $. Let $\mathcal{U}$ be any regular open cover of $X$
and $K\in H\left( X,f\right) $. Since $f:K\rightarrow K$ is bijective and $R$%
-map, then $f\left( K\right) =K$. By Theorem 2.1(d) in \cite{gulam}, $%
M_{K}\left( f^{-1}\left( \mathcal{U}\right) \right) =M_{K}\left( \mathcal{U}%
\right) $. So,%
\begin{eqnarray*}
Ent_{N}\left( f^{-1},\mathcal{U},K\right) &=&\lim_{n\rightarrow \infty }%
\frac{1}{n}M_{K}\left( \bigvee_{i=0}^{n-1}f^{i}\left( \mathcal{U}\right)
\right) \\
&=&\lim_{n\rightarrow \infty }\frac{1}{n}M_{K}\left( f^{-\left( n-1\right)
}\left( \bigvee_{i=0}^{n-1}f^{i}\left( \mathcal{U}\right) \right) \right) \\
&=&\lim_{n\rightarrow \infty }\frac{1}{n}M_{K}\left(
\bigvee_{i=0}^{n-1}f^{-i}\left( \mathcal{U}\right) \right) \\
&=&Ent_{N}\left( f,\mathcal{U},K\right) .
\end{eqnarray*}%
Then, we have $Ent_{N}\left( f\right) \leq Ent_{N}\left( f^{-1}\right) $ by
Definitions \ref{def_N1} and \ref{def_N2}.

By interchange the role of $f$ and $f^{-1}$ in the preceding argument, we
will have $Ent_{N}\left( f^{-1}\right) \leq Ent_{N}\left( \left(
f^{-1}\right) ^{-1}\right) =Ent_{N}\left( f\right) $. Thus, $Ent_{N}\left(
f\right) =Ent_{N}\left( f^{-1}\right) $.
\end{proof}

\section{Topological Nearly Entropy on Nearly Compact Spaces}

In this section, we define topological nearly entropy on nearly compact
space by using $R$-map. Moreover, we show that this definition is equivalent
to the previous definition in, see Definition \ref{def_N2} whenever $X$ is
nearly compact.

\begin{definition}
\label{topo n} Let $\left( X,f\right) $ be a nearly compact $R$-dynamical
system and $\mathcal{U}$ be a regular open cover of $X$. Let%
\begin{equation*}
N_{n}\left( \mathcal{U}\right) =\min \left\{ \limfunc{card}\left( \mathcal{V}%
\right) :\mathcal{V}\text{ is a subcover of }\mathcal{U},X=\bigcup_{V\in
\mathcal{V}}V\right\} .
\end{equation*}%
Since $X$ is nearly compact, $N_{n}\left( \mathcal{U}\right) $ is a positive
integer. Let $M_{n}\left( \mathcal{U}\right) =\log N_{n}\left( \mathcal{U}%
\right) .$
\end{definition}

\begin{theorem}
\label{thm_n2.1} Let $\left( X,f\right) $ be a nearly compact $R$-dynamical
system and $\mathcal{U}$ and $\mathcal{V}$ be a regular open covers of $X$.
Then the following statements hold:

\begin{enumerate}
\item[$\left( a\right) $] $M_{n}\left( \mathcal{U}\right) \geq 0$.

\item[$\left( b\right) $] $\mathcal{U}\prec \mathcal{V}$ implies $%
M_{n}\left( \mathcal{U}\right) \leq M_{n}\left( \mathcal{V}\right) $.

\item[$\left( c\right) $] $M_{n}\left( \mathcal{U}\vee \mathcal{V}\right)
\leq M_{n}\left( \mathcal{U}\right) +M_{n}\left( \mathcal{V}\right) $.

\item[$\left( d\right) $] $M_{n}\left( f^{-1}\left( \mathcal{U}\right)
\right) \leq M_{n}\left( \mathcal{U}\right) $. The equality holds when $f$
is onto.
\end{enumerate}
\end{theorem}

\begin{proof}
The proofs are almost analogous to the proof of Theorem \ref{Th2.1} (see
\cite{gulam}) by replacing the role of $K$ as $X$ and slight modification of
the proof, one can easily establish this theorem.
\end{proof}

\begin{theorem}
\label{lim} Let $\left( X,f\right) $ be a nearly compact $R$-dynamical
system and $\mathcal{U}$ be a regular open cover of $X$. Then
\begin{equation*}
\lim_{m\rightarrow \infty }\frac{1}{m}M_{n}\left(
\bigvee_{i=0}^{m-1}f^{-i}\left( \mathcal{U}\right) \right)
\end{equation*}%
exists.
\end{theorem}

\begin{proof}
The proof are almost analogous to the proof af Theorem 2.3 in \cite{gulam}
by replacing the role of $K$ as $X$ and slight modification of the proof,
one can easily establish this theorem.
\end{proof}

\begin{definition}
\label{def topo N} Let $\left( X,f\right) $ be a nearly compact $R$%
-dynamical system and $\mathcal{U}$ be a regular open cover of $X$. Then
\begin{equation*}
Ent_{n}\left( f,\mathcal{U}\right) =\lim_{m\rightarrow \infty }\frac{1}{m}%
M_{n}\left( \bigvee_{i=0}^{m-1}f^{-i}\left( \mathcal{U}\right) \right)
\end{equation*}%
is called topological nearly entropy of $f$ relative to $\mathcal{U}$ and
\begin{equation*}
Ent_{n}\left( f\right) =\sup_{\mathcal{U}}\left\{ Ent_{n}\left( f,\mathcal{U}%
\right) :\mathcal{U}\mathrm{\ }\text{is~a regular~open~of}~X\right\}
\end{equation*}%
is called the topological nearly entropy of $f$.
\end{definition}

In the next theorem, we show the relationship between this new notion with
the previous Definition \ref{def_N1}.

\begin{theorem}
\label{n=N} Let $\left( X,f\right) $ be a topological $R$-dynamical system, $%
K\in H\left( X,f\right) $ and $\mathcal{U}$ be any regular open cover of $X$%
. Then for the subspace $K$, $Ent_{n}\left( f|_{K},\mathcal{U}|_{K}\right)
=Ent_{N}\left( f,\mathcal{U},K\right) $, where $\mathcal{U}|_{K}=\left\{
U\cap K:U\in \mathcal{U}\right\} $ and mapping $f|_{K}:K\rightarrow K$ is
defined by $f|_{K}\left( k\right) =f\left( k\right) $ for each $k\in K$.
\end{theorem}

\begin{proof}
Let $N_{n}\left( \bigvee_{i=0}^{m-1}\left( f|_{K}\right) ^{-i}\left(
\mathcal{U}|_{K}\right) \right) =s$. Then there exist a finite subfamily $%
\left\{ U_{1},U_{2},\ldots ,U_{s}\right\} $ of $\bigvee_{i=0}^{m-1}$ $\left(
f|_{K}\right) ^{-i}\left( \mathcal{U}|_{K}\right) $ such that $K\subseteq
\bigcup_{l=1}^{s}U_{l}$. Then, for each $l\in \left\{ 1,2,\ldots ,s\right\} $%
, there exist $\left\{ U_{l}^{0},U_{l}^{1},U_{l}^{2},\ldots
,U_{l}^{m-1}\right\} \in \mathcal{U}$ such that $U_{l}=\bigcap_{i=0}^{m-1}%
\left( f|_{K}\right) ^{-i}\left( U_{l}^{i}\cap K\right) $. Clearly, $\left\{
U_{l}^{^{\prime }}:U_{l}^{^{\prime }}=\bigcap_{i=0}^{m-1}f^{-i}\left(
U_{l}^{i}\right) \right\} $ is a subfamily of $\bigvee_{i=0}^{m-1}f^{-i}%
\left( \mathcal{U}\right) $ with $U_{l}\subset U_{l}^{^{\prime }}$.
Moreover,
\begin{equation*}
K=\left( \bigcup_{l=1}^{s}U_{l}\right) \cap K=\bigcup_{l=1}^{s}\left(
U_{l}\cap K\right) \subseteq \bigcup_{l=1}^{s}\left( U_{l}^{^{\prime }}\cap
K\right) \subseteq \bigcup_{l=1}^{s}U_{l}^{^{\prime }}.
\end{equation*}%
Thus,
\begin{equation*}
N_{K}\left( \bigvee_{i=0}^{m-1}f^{-i}\left( \mathcal{U}\right) \right) \leq
s=N_{n}\left( \bigvee_{i=0}^{m-1}\left( f|_{K}\right) ^{-i}\left( \mathcal{U}%
|_{K}\right) \right)
\end{equation*}%
and so $M_{K}\left( \bigvee_{i=0}^{m-1}f^{-i}\left( \mathcal{U}\right)
\right) \leq M_{n}\left( \bigvee_{i=0}^{m-1}\left( f|_{K}\right) ^{-i}\left(
\mathcal{U}|_{K}\right) \right) $. Hence,%
\begin{eqnarray*}
Ent_{N}\left( f,\mathcal{U},K\right) &=&\underset{m\rightarrow \infty }{\lim
}\frac{1}{m}M_{K}\left( \dbigvee\limits_{i=0}^{m-1}f^{-i}\left( \mathcal{U}%
\right) \right) \\
&\leq &\lim_{m\rightarrow \infty }\frac{1}{m}M_{n}\left(
\bigvee_{i=0}^{m-1}\left( f|_{K}\right) ^{-i}\left( \mathcal{U}|_{K}\right)
\right) \\
&=&Ent_{n}\left( f|_{K},\mathcal{U}|_{K}\right) .
\end{eqnarray*}

To establish the reverse inequality, consider $N_{K}\left(
\bigvee_{i=0}^{m-1}f^{-i}\left( \mathcal{U}\right) \right) =q$. Then there
exists a finite subfamily $\left\{ V_{1},V_{2},\ldots ,V_{q}\right\} $ of $%
\bigvee_{i=0}^{m-1}f^{-i}\left( \mathcal{U}\right) $ such that $K\subseteq
\bigcup_{j=1}^{q}V_{j}$. Then for each $j\in \left\{ 1,2,\ldots ,q\right\} $%
, there exist $\left\{ V_{j}^{0},V_{j}^{1},V_{j}^{2},\ldots
,V_{j}^{m-1}\right\} $ $\in \mathcal{U}$ such that $V_{j}=%
\bigcap_{i=0}^{m-1}f^{-i}\left( V_{j}^{i}\right) $. Consider the subfamily
\begin{equation*}
\left\{ V_{j}^{^{\prime }}:V_{j}^{^{\prime }}=\bigcap_{i=0}^{m-1}\left(
f|_{K}\right) ^{-i}\left( V_{j}^{i}\cap K\right) \right\}
\end{equation*}%
of $\bigvee_{i=0}^{m-1}\left( f|_{K}\right) ^{-i}\left( \mathcal{U}%
|_{K}\right) $. Since $K\subseteq \bigcup_{j=1}^{q}V_{j}$, then $%
\bigcup_{j=1}^{q}\left( V_{j}\cap K\right) =K$. Again, $f\left( K\right)
\subseteq K$ ensure that $K\subseteq f^{-i}\left( K\right) $ for each $%
i=0,1,2,\ldots $, and
\begin{equation*}
V_{j}\cap K=\bigcap_{i=0}^{m-1}\left( f^{-i}\left( V_{j}^{i}\right) \cap
K\right) \subseteq \bigcap_{i=0}^{m-1}\left( f^{-i}\left( V_{j}^{i}\right)
\cap f^{-i}\left( K\right) \right) =\bigcap_{i=0}^{m-1}f^{-i}\left(
V_{j}^{i}\cap K\right) .
\end{equation*}%
Thus,
\begin{equation*}
V_{j}\cap K=\left( \bigcap_{i=0}^{m-1}f^{-i}\left( V_{j}^{i}\cap K\right)
\right) \cap K=\bigcap_{i=0}^{m-1}\left( f|_{K}\right) ^{-i}\left(
V_{j}^{i}\cap K\right) =V_{j}^{^{\prime }}
\end{equation*}%
and so
\begin{equation*}
K=\bigcup_{j=1}^{q}\left( V_{j}\cap K\right)
=\bigcup_{j=1}^{q}V_{j}^{^{\prime }}.
\end{equation*}%
Hence, $N_{n}\left( \bigvee_{i=0}^{m-1}\left( f|_{K}\right) ^{-i}\left(
\mathcal{U}|_{K}\right) \right) \leq q=N_{K}\left(
\bigvee_{i=0}^{m-1}f^{-i}\left( \mathcal{U}\right) \right) $ and so,
\begin{equation*}
M_{n}\left( \bigvee_{i=0}^{m-1}\left( f|_{K}\right) ^{-i}\left( \mathcal{U}%
|_{K}\right) \right) \leq M_{K}\left( \bigvee_{i=0}^{m-1}f^{-i}\left(
\mathcal{U}\right) \right) .
\end{equation*}%
Therefore,
\begin{eqnarray*}
Ent_{n}\left( f|_{K},\mathcal{U}|_{K}\right) &=&\lim_{m\rightarrow \infty }%
\frac{1}{m}M_{n}\left( \bigvee_{i=0}^{m-1}\left( f|_{K}\right) ^{-i}\left(
\mathcal{U}|_{K}\right) \right) \\
&\leq &\lim_{m\rightarrow \infty }\frac{1}{m}M_{K}\left(
\bigvee_{i=0}^{m-1}f^{-i}\left( \mathcal{U}\right) \right) \\
&=&Ent_{N}\left( f,\mathcal{U},K\right) .
\end{eqnarray*}%
Thus, $Ent_{n}\left( f|_{K},\mathcal{U}|_{K}\right) =Ent_{N}\left( f,%
\mathcal{U},K\right) .$
\end{proof}

The following theorem indicates that the concept of topological nearly
entropy $Ent_{N}\left( f\right) $ is coincides with $Ent_{n}\left( f\right) $
when $X$ itself is nearly compact.

\begin{theorem}
Let $\left( X,f\right) $ be a nearly compact $R$-dynamical system. Then $%
Ent_{n}\left( f\right) =Ent_{N}\left( f\right) $.
\end{theorem}

\begin{proof}
It is clear that $X\in H\left( X,f\right) $ and by Definition \ref{def_N2},%
\begin{equation*}
Ent_{N}\left( f\right) =\sup_{K}\left\{ Ent_{N}\left( f,K\right) :K\in
H\left( X,f\right) \right\} =Ent_{N}\left( f,X\right) .
\end{equation*}%
Again for any regular open cover $\mathcal{U}$ of $X$, Theorem \ref{n=N}
ensures that $Ent_{n}\left( f,\mathcal{U}\right) =Ent_{N}\left( f,\mathcal{U}%
,X\right) $ and so%
\begin{eqnarray*}
Ent_{N}\left( f,X\right) &=&\sup_{\mathcal{U}}\left\{ Ent_{N}\left( f,%
\mathcal{U},X\right) :\mathcal{U}~\text{is~a~regular~open~cover~of}~X\right\}
\\
&=&\sup_{\mathcal{U}}\left\{ Ent_{n}\left( f,\mathcal{U}\right) :\mathcal{U}~%
\text{is~a~regular~open~cover~of}~X\right\} \\
&=&Ent_{n}\left( f\right) .
\end{eqnarray*}%
Thus, $Ent_{N}\left( f\right) =Ent_{N}\left( f,X\right) =Ent_{n}\left(
f\right) $.
\end{proof}

Next , we study some properties of of topological nearly entropy on nearly
compact and Hausdorff space by introducing $R$-space. So, the result
demonstrate that topological nearly entropy of a map $f$ is coincide with
topological nearly entropy of it restriction $f|_{K}$ whenever $K\in H\left(
X,f\right) $.

Recall that in general, the union of two regular open sets is not regular
open. For example, let $\left( 1,2\right) $ and $\left( 2,3\right) $ be two
regular open sets in $\mathbb{R}$. Hence, the union of both intervals are
not regular open since $\limfunc{int}\left( \limfunc{cl}\left( \left(
1,2\right) \cup \left( 2,3\right) \right) \right) =\left( 1,3\right) $ but $%
\left( 1,3\right) \neq \left( 1,2\right) \cup \left( 2,3\right) $. As an
upshot, we introduce $R$-space as follows.

\begin{definition}
A topological space $X$ is said to be $R$-space if the union of each regular
open subset of $X$ is also regular open.
\end{definition}

Now, by using $R$-space we study some basic properties of nearly compact in
Hausdorff space.

\begin{lemma}
\label{lem haus} Let $X$ be a Hausdorff and $R$-space. If $A$ is a nearly
compact relative to $X$ and $p\notin A$, then there exist regular open sets $%
G,H$ of $X$ such that $p\in G,A\subset H$ and $G\cap H=\emptyset .$
\end{lemma}

\begin{proof}
Let $a\in A$. Since $p\notin A$, $p\neq a$. By hypothesis, $X$ is Hausdorff,
hence there exist open sets $G_{a},H_{a}$ of $X$ such that $p\in G_{a},a\in
H_{a}$ and $G_{a}\cap H_{a}=\emptyset .$ Since $A\subset \bigcup \left\{
H_{a}:a\in A\right\} $, then $\left\{ H_{a}:a\in A\right\} $ forms a cover
of $A$ by open subsets of $X$. As $A$ is nearly compact relative to $X$, so
there exist a finite subfamily $H_{a_{1}},\ldots ,H_{a_{m}}$ of $\left\{
H_{a}\right\} $ such that $A\subset \bigcup_{i=1}^{m}\limfunc{int}\left(
\limfunc{cl}\left( H_{a_{i}}\right) \right) $.

Let $H=\limfunc{int}\left( \limfunc{cl}\left( H_{a_{1}}\right) \right) \cup
\cdots \cup \limfunc{int}\left( \limfunc{cl}\left( H_{a_{m}}\right) \right) $
and $G=\limfunc{int}\left( \limfunc{cl}\left( G_{a_{1}}\right) \right) \cap
\cdots \cap \limfunc{int}\left( \limfunc{cl}\left( G_{a_{m}}\right) \right) $%
. Since $X$ is $R$-space, then $H$ and $G$ are regular open. Furthermore, $%
A\subset H$ and $p\in G$ since $p$ belongs to each $G_{a_{i}}$ individually.
Lastly we claim that $G\cap H=\emptyset $. Note first that, $\limfunc{int}%
\left( \limfunc{cl}\left( G_{a_{i}}\right) \right) $ $\cap $ $\limfunc{int}%
\left( \limfunc{cl}\left( H_{a_{i}}\right) \right) =\emptyset $ implies $%
G\cap \limfunc{int}\left( \limfunc{cl}\left( H_{a_{i}}\right) \right)
=\emptyset $. Thus, by the distributive law,%
\begin{eqnarray*}
G\cap H &=&G\cap \left( \limfunc{int}\left( \limfunc{cl}\left(
H_{a_{1}}\right) \right) \cup \cdots \cup \limfunc{int}\left( \limfunc{cl}%
\left( H_{a_{m}}\right) \right) \right) \\
&=&\left( G\cap \limfunc{int}\left( \limfunc{cl}\left( H_{a_{1}}\right)
\right) \right) \cup \cdots \cup \left( G\cap \limfunc{int}\left( \limfunc{cl%
}\left( H_{a_{m}}\right) \right) \right) \\
&=&\emptyset \cup \cdots \cup \emptyset \\
&=&\emptyset .
\end{eqnarray*}%
This completes the proof.
\end{proof}

\begin{lemma}
\label{lem haus2} Let $X$ be a Hausdorff and $R$-space. If $A$ is a nearly
compact relative to $X$ and $p\notin A$, then there is a regular open set $G$
such that $p\in G\subset A^{c}$.
\end{lemma}

\begin{proof}
By Lemma \ref{lem haus}, there exist regular open sets $G$ and $H$ in $X$
such that $p\in G,A\subset H$ and $G\cap H=\emptyset $. Hence $G\cap
A=\emptyset $, and thus $p\in G\subset A^{c}.$
\end{proof}

\begin{lemma}
\label{lem haus3} Let $X$ be a Hausdorff and $R$-space. If $A$ is a nearly
compact relative to $X$, then $A$ is a regular closed subset of $X$.
\end{lemma}

\begin{proof}
We prove, equivalently, that $A^{c}$ is a regular open subset of $X$. Let $%
p\in A^{c}$, i.e., $p\notin A$. Then by Lemma \ref{lem haus2} there exists a
regular open set $G_{p}$ of $X$ such that $p\in G_{p}\subset A^{c}$. Hence $%
A^{c}=\bigcup_{p\in A^{c}}G_{p}$. Thus $A^{c}$ is regular open as it is the
union of regular open sets in $R$-space $X$. Therefore $A$ is a regular
closed subset of $X$.
\end{proof}

The next theorem shows that the conditions on which the topological nearly
entropy of $f$ and it restriction $f|_{K}$ coincides.

\begin{theorem}
Let $\left( X,f\right) $ be a topological $R$-dynamical system. Assume that $%
X$ is Hausdorff and $R$-space. If $K\in H\left( X,f\right) $ then

$\left( a\right) $ $Ent_{N}\left( f,K\right) =Ent_{n}\left( f|_{K},K\right) $

$\left( b\right) $ $Ent_{N}\left( f\right) =Ent_{n}\left( f|_{K}\right) ,$

where $f|_{K}$ is defined as in the Theorem \ref{n=N}.
\end{theorem}

\begin{proof}
$\left( a\right) $ Let $\mathcal{U}_{K}$ be any regular open cover of $K$
(by regular open subsets of $K$). Since $X$ is Hausdorff and $K$ is nearly
compact relative to $X$, then by Lemma \ref{lem haus3}, $K$ is a regular
closed subset of $X$. For every $A\in \mathcal{U}_{K}$, there exist a
regular open subset $U_{A}$ of $X$ satisfying $A=U_{A}\cap K$. Denote $%
\mathcal{U}^{\prime }=\left\{ U_{A}:A\in \mathcal{U}_{K}\right\} $. Clearly,
$\mathcal{U}=\mathcal{U}^{\prime }\cup \left\{ X\setminus K\right\} $ is a
regular open cover of $X$ satisfying $\mathcal{U}|_{K}=\mathcal{U}_{K}\cup
\left\{ \emptyset \right\} $. Hence, every regular open cover $\mathcal{U}%
_{K}$ of $K$ is a restriction $\mathcal{U}|_{K}$ of some special regular
open cover $\mathcal{U}$ of $X$ that includes regular open subset $%
X\setminus K$ of $X$. From Definition \ref{def topo N}, we have
\begin{equation}
Ent_{n}\left( f|_{K},K\right) =\sup_{\mathcal{U}_{K}}\left\{ Ent_{n}\left(
f|_{K},\mathcal{U}_{K},K\right) \right\} =\sup_{\mathcal{U}}\left\{
Ent_{n}\left( f|_{K},\mathcal{U}|_{K},K\right) \right\} .  \label{2}
\end{equation}%
By Definition \ref{def_N1}, we have
\begin{equation*}
Ent_{N}\left( f,K\right) =\sup_{\beta }\left\{ Ent_{N}\left( f,\beta
,K\right) :\beta \text{ runs over all regular open covers of }X\right\} .
\end{equation*}%
It follows from Theorem \ref{n=N} that, $Ent_{N}\left( f,\mathcal{U}%
,K\right) =Ent_{n}\left( f|_{K},\mathcal{U}|_{K},K\right) $, thus $%
Ent_{n}\left( f|_{K},K\right) =\sup_{\mathcal{U}}\left\{ Ent_{n}\left(
f|_{K},\mathcal{U}|_{K},K\right) \right\} =\sup_{\mathcal{U}}$ $\left\{
Ent_{N}\left( f,\mathcal{U},K\right) \right\} $. Recall that the regular
open covers $\mathcal{U}$ are some special regular open covers of $X$.
Hence,
\begin{equation*}
\sup_{\mathcal{U}}\left\{ Ent_{N}\left( f,\mathcal{U},K\right) \right\} \leq
Ent_{N}\left( f,K\right)
\end{equation*}%
implying
\begin{equation*}
Ent_{n}\left( f|_{K},K\right) \leq Ent_{N}\left( f,K\right) .
\end{equation*}

For reverse inequality, i.e., $Ent_{n}\left( f|K,K\right) \geq Ent_{N}\left(
f,K\right) $, let $\mathcal{U}$ be any regular open cover of $X$. From
Theorem \ref{n=N}, we have $Ent_{N}\left( f,\mathcal{U},K\right)
=Ent_{n}\left( f|_{K},\mathcal{U}|_{K},K\right) $ and thus from Definition %
\ref{def_N1},%
\begin{equation*}
Ent_{N}\left( f,K\right) =\sup_{\mathcal{U}}\left\{ Ent_{N}\left( f,\mathcal{%
U},K\right) \right\} =\sup_{\mathcal{U}}\left\{ Ent_{n}\left( f|_{K},%
\mathcal{U}|_{K},K\right) \right\} .
\end{equation*}%
As $\mathcal{U}|_{K}$ are some special regular open covers of $K$, we have%
\begin{equation*}
\sup_{\mathcal{U}}\left\{ Ent_{n}\left( f|_{K},\mathcal{U}|_{K},K\right)
\right\} \leq Ent_{n}\left( f|_{K},K\right) .
\end{equation*}%
Thus,%
\begin{eqnarray*}
Ent_{N}\left( f,K\right) &=&\sup_{\mathcal{U}}\left\{ Ent_{N}\left( f,%
\mathcal{U},K\right) \right\} \\
&=&\sup_{\mathcal{U}}\left\{ Ent_{n}\left( f|_{K},\mathcal{U}|_{K},K\right)
\right\} \\
&\leq &Ent_{n}\left( f|_{K},K\right) .
\end{eqnarray*}

$\left( b\right) $ By Definition \ref{def_N2},%
\begin{equation*}
Ent_{N}\left( f\right) =\sup_{K}\left\{ Ent_{N}\left( f,K\right) :K\in
H\left( X,f\right) \right\} .
\end{equation*}%
Now let $\mathcal{U}$ be any regular open cover of $X$. From Definition \ref%
{def_N1},
\begin{equation*}
Ent_{N}\left( f,K\right) =\sup_{\mathcal{U}}\left\{ Ent_{N}\left( f,\mathcal{%
U},K\right) :\mathcal{U}\mathrm{\ }\text{is~a regular~open~cover of}%
~X\right\} .
\end{equation*}%
Thus,%
\begin{eqnarray*}
Ent_{N}\left( f\right) &=&\sup_{K}\left\{ \sup_{\mathcal{U}}\left\{
Ent_{N}\left( f,\mathcal{U},K\right) :\mathcal{U}\mathrm{\ }\text{is~a
regular~open~cover of}~X\right\} :K\in H\left( X,f\right) \right\} \\
&=&\sup_{\mathcal{U}}\left\{ \sup_{K}\left\{ Ent_{N}\left( f,\mathcal{U}%
,K\right) :K\in H\left( X,f\right) \right\} :\mathcal{U}\mathrm{\ }\text{%
is~a regular~open~cover of}~X\right\} \\
&=&\sup_{\mathcal{U}}\left\{ Ent_{N}\left( f,\mathcal{U}\right) :\text{ }%
\mathcal{U}\mathrm{\ }\text{is~a regular~open~cover of}~X\right\}
\end{eqnarray*}%
By Theorem \ref{n=N},
\begin{equation*}
Ent_{N}\left( f,\mathcal{U}\right) =Ent_{n}\left( f|_{K},\mathcal{U}%
|_{K}\right) .
\end{equation*}%
Hence,%
\begin{eqnarray*}
Ent_{N}\left( f\right) &=&\sup_{\mathcal{U}}\left\{ Ent_{n}\left( f|_{K},%
\mathcal{U}|_{K}\right) :\text{ }\mathcal{U}\mathrm{\ }\text{is~a
regular~open~cover of}~X\right\} \\
&=&\sup_{\mathcal{U}_{K}}\left\{ Ent_{n}\left( f|_{K},\mathcal{U}%
_{K},K\right) :\mathcal{U}_{K}\mathrm{\ }\text{is~a regular~open~cover of}%
~K\right\} \\
&=&Ent_{n}\left( f|_{K}\right) .
\end{eqnarray*}%
The second equality follows from equation $\left( \ref{2}\right) $. This
completes the proof.
\end{proof}

\section{Topological Nearly Entropy of Product Space}

Let $\left( X,f\right) $ and $\left( Y,h\right) $ be two topological $R$%
-dynamical systems. For the product space $X\times Y$, define function $%
f\times h:X\times Y\rightarrow X\times Y$ by $\left( f\times h\right) \left(
x,y\right) =\left( f\left( x\right) ,h\left( y\right) \right) $. This
function $f\times h$ is $R$-map and $\left( X\times Y,f\times h\right) $
forms a topological $R$-dynamical system. If $\mathcal{U}$ and $\mathcal{V}$
are regular open covers of $X$ and $Y$, respectively, then $\mathcal{U}%
\times \mathcal{V}$ is a regular open cover of $X\times Y$.

As a consequence, we obtain the properties of topological nearly entropy for
product space. But, firstly we show that $R$-map preserve nearly compactness
relative to the space.

\begin{lemma}
\label{fA} Let $f:X\rightarrow Y$ be an $R$-map. If $A$ is nearly compact
relative to $X$, then $f\left( A\right) $ is nearly compact relative to $Y$.
\end{lemma}

\begin{proof}
Let $\left\{ V_{\alpha }:\alpha \in \Delta \right\} $ be a cover of $f\left(
A\right) $ by regular open subsets of $Y$. For each $x\in A$, take $\alpha
_{x}\in \Delta $ such that $f\left( x\right) \in V_{\alpha _{x}}$. Since $f$
is $R$-map, $f^{-1}\left( V_{\alpha _{x}}\right) $ is a regular open subset
of $X$ containing $x$. Now $\left\{ f^{-1}\left( V_{\alpha _{x}}\right)
:x\in A\right\} $ forms a regular open cover of $A$ by regular open subsets
of $X$. Since $A$ is nearly compact relative to $X$, there exists a finite
subset of points $x_{1},\ldots ,x_{n}$ of $A$ such that $A\subseteq
\bigcup_{k=1}^{n}f^{-1}\left( V_{\alpha _{x_{k}}}\right) $. Now,
\begin{equation*}
f\left( A\right) \subseteq f\left( \bigcup_{k=1}^{n}f^{-1}\left( V_{\alpha
_{x_{k}}}\right) \right) =\bigcup_{k=1}^{n}f\left( f^{-1}\left( V_{\alpha
_{x_{k}}}\right) \right) \subseteq \bigcup_{k=1}^{n}V_{\alpha _{x_{k}}}.
\end{equation*}%
Therefore $f\left( A\right) $ is nearly compact relative to $Y$.
\end{proof}

\begin{lemma}
\label{cover} Let $X$ and $Y$ be two nearly compact $R$-dynamical system.
Suppose that $\mathcal{U}$ and $\mathcal{V}$ are regular open covers of $X$
and $Y$, respectively. Then for a regular open cover $\mathcal{U}\times
\mathcal{V}$ of $X\times Y$, we have%
\begin{equation*}
M_{n}\left( \mathcal{U}\times \mathcal{V}\right) \leq M_{n}\left( \mathcal{U}%
\right) +M_{n}\left( \mathcal{V}\right) .
\end{equation*}
\end{lemma}

\begin{proof}
Let $\left\{ U_{1},\ldots ,U_{k}\right\} $ and $\left\{ V_{1},\ldots
,V_{p}\right\} $ be subcovers of $\mathcal{U}$ and $\mathcal{V}$ of minimal
cardinality, respectively. Then $\left\{ U_{i}\times V_{j}:i=1,\ldots
,k,j=1,\ldots ,p\right\} $ is a subcover of $\mathcal{U}\times \mathcal{V}$
with cardinality $kp$. It follows that%
\begin{equation*}
N_{n}\left( \mathcal{U}\times \mathcal{V}\right) \leq kp=N_{n}\left(
\mathcal{U}\right) \cdot N_{n}\left( \mathcal{V}\right) .
\end{equation*}%
Thus,%
\begin{eqnarray*}
M_{n}\left( \mathcal{U}\times \mathcal{V}\right) &=&\log N_{n}\left(
\mathcal{U}\times \mathcal{V}\right) \\
&\leq &\log \left( N_{n}\left( \mathcal{U}\right) \cdot N_{n}\left( \mathcal{%
V}\right) \right) \\
&=&\log N_{n}\left( \mathcal{U}\right) +\log N_{n}\left( \mathcal{V}\right)
\\
&=&M_{n}\left( \mathcal{U}\right) +M_{n}\left( \mathcal{V}\right) .
\end{eqnarray*}
\end{proof}

\begin{lemma}
\label{lem 0.1} Let $\left( X,f\right) $ and $\left( Y,h\right) $ be two
topological $R$-dynamical systems. Suppose $T_{x}:X\times Y\rightarrow X$
and $T_{y}:X\times Y\rightarrow Y$ are the $R$-map projections on $X$ and $Y$%
, respectively. If $K\in H\left( X\times Y,f\times h\right) $, then $%
T_{x}\left( K\right) \in H\left( X,f\right) $, $T_{y}\left( K\right) \in
H\left( Y,h\right) $ and $K\subseteq T_{x}\left( K\right) \times T_{y}\left(
K\right) $.
\end{lemma}

\begin{proof}
Since $T_{x}$ and $T_{y}$ are $R$-maps and $K\in H\left( X\times Y,f\times
h\right) $, $T_{x}\left( K\right) $ is nearly compact relative to $X$ and $%
T_{y}\left( K\right) $ is nearly compact relative to $Y$ by Lemma \ref{fA}.
It suffices to show that $T_{x}\left( K\right) $ and $T_{y}\left( K\right) $
are invariant nearly compact relative to $X$ and $Y$ by $f$ and $h$,
respectively. For any $x\in T_{x}\left( K\right) $, there exists a $y\in Y$
which satisfying $\left( x,y\right) \in K$. Since $K$ is an invariant nearly
compact relative to $X\times Y$ by $f\times h$, we have
\begin{equation*}
\left( f\times h\right) \left( x,y\right) =\left( f\left( x\right) ,h\left(
y\right) \right) \in K.
\end{equation*}%
Thus
\begin{equation*}
T_{x}\left( f\left( x\right) ,h\left( y\right) \right) =f\left( x\right) \in
T_{x}\left( K\right) .
\end{equation*}%
Hence $f\left( T_{x}\left( K\right) \right) \subseteq T_{x}\left( K\right) $
and therefore $T_{x}\left( K\right) \in H\left( X,f\right) $.

There also exists an $x\in X$ for any $y\in T_{y}\left( K\right) $
satisfying $\left( x,y\right) \in K$. Since $K$ is an invariant nearly
compact relative to $X\times Y$ by $f\times h$, we have
\begin{equation*}
\left( f\times h\right) \left( x,y\right) =\left( f\left( x\right) ,h\left(
y\right) \right) \in K.
\end{equation*}%
Thus
\begin{equation*}
T_{y}\left( f\left( x\right) ,h\left( y\right) \right) =h\left( y\right) \in
T_{y}\left( K\right) .
\end{equation*}%
Hence $h\left( T_{y}\left( K\right) \right) \subseteq T_{y}\left( K\right) $
and therefore $T_{y}\left( K\right) \in H\left( Y,h\right) .$

Finally, if $\left( x,y\right) \in K$, then by the definitions of $%
T_{x}\left( K\right) $ and $T_{y}\left( K\right) $, we will have $x\in
T_{x}\left( K\right) $ and $y\in T_{y}\left( K\right) $ and thus $\left(
x,y\right) \in $ $T_{x}\left( K\right) \times T_{y}\left( K\right) $.
\end{proof}

\begin{lemma}
\label{lem 1}\cite{singal} The product of two nearly compact spaces is
nearly compact.
\end{lemma}

\begin{lemma}
\label{lem 0.3} Let $\left( X,f\right) $ and $\left( Y,h\right) $ be two
nearly compact $R$-dynamical systems. If $\mathcal{W}$ is a regular open
cover of $X\times Y$, then there exist regular open cover $\mathcal{U}$ of $%
X $ and regular open cover $\mathcal{V}$ of $Y$ satisfying $\mathcal{W}\prec
\mathcal{U}\times \mathcal{V}$.
\end{lemma}

\begin{proof}
Since $X$ and $Y$ are nearly compact, by Lemma \ref{lem 1} the product space
of $X\times Y$ is also nearly compact and for each $x\in X$, $\left\{
x\right\} \times Y$ is a nearly compact subset of $X\times Y$. Since $%
\mathcal{W}$ is a regular open cover of $X\times Y$, then there exist $%
A_{y}\in \mathcal{W}$ for any $y\in Y$ which satisfying $\left( x,y\right)
\in A_{y}$. Let $U\left( x,y\right) $ be a regular open neighborhood of $x$
in $X$ and $V\left( x,y\right) $ be a regular open neigborhood of $y$ in $Y$
with%
\begin{equation*}
U\left( x,y\right) \times V\left( x,y\right) \subseteq A_{y}.
\end{equation*}%
Since $Y$ is nearly compact, the regular open cover $\left\{ V\left(
x,y\right) :y\in Y\right\} $ of $Y$ has a finite subcover, say $\left\{
V\left( x,y_{1}\right) ,\ldots ,V\left( x,y_{m\left( x\right) }\right)
\right\} $. Hence, for each $U\left( x,y_{j}\right) \times V\left(
x,y_{j}\right) ,j\in \left\{ 1,2,\ldots ,m\left( x\right) \right\} $, there
exist an $A_{j}\in \mathcal{W}$ such that $U\left( x,y_{j}\right) \times
V\left( x,y_{j}\right) \subseteq A_{j}$. Denote $U\left( x\right)
=\dbigcap\limits_{j=1}^{m\left( x\right) }U\left( x,y_{j}\right) $, then%
\begin{equation*}
U\left( x\right) \times V\left( x,y_{j}\right) \subseteq A_{j}.
\end{equation*}%
On the other hand, as $X$ is nearly compact, the regular open cover $\left\{
U\left( x\right) :x\in X\right\} $ of $X$ has a finite subcover, say $%
\left\{ U\left( x_{1}\right) ,\ldots ,U\left( x_{n}\right) \right\} $.
Observe that each $U\left( x_{i}\right) \times V\left( x_{i},y_{j}\right)
,i\in \left\{ 1,2,\ldots ,n\right\} ,j\in \left\{ 1,2,\ldots ,m\left(
x_{i}\right) \right\} $ is contained in an element of $\mathcal{W}$. For any
$y\in Y$, let $V\left( y\right) $ be the intersection of those elements of
the collection $\left\{ V\left( x_{i},y_{j}\right) :i=1,\ldots ,n;j=1,\ldots
,m\left( x_{i}\right) \right\} $ that contain $y$. Since this collection has
finitely many elements and each of them is regular open subset of $Y$, then $%
V\left( y\right) $ is also regular open subset of $Y$. Notice that $\left\{
V\left( y\right) :y\in Y\right\} $ is a regular open cover of $Y$ and hence
it has a finite subcover as $Y$ is nearly compact. Denote this subcover by $%
\left\{ V\left( y_{1}\right) ,\ldots ,V\left( y_{m}\right) \right\} $. Now, $%
\left\{ U\left( x_{i}\right) \times V\left( y_{j}\right) :i=1,\ldots
,n;j=1,\ldots ,m\right\} $ is a regular open cover of $X\times Y$ and each $%
U\left( x_{i}\right) \times V\left( y_{j}\right) $ is contained in an
element of $\mathcal{W}$. Let $\mathcal{U}=\left\{ U\left( x_{i}\right)
:i=1,\ldots ,n\right\} $ and $\mathcal{V}=\left\{ V\left( y_{j}\right)
:j=1,\ldots ,m\right\} $. Then $\mathcal{U}\times \mathcal{V}$ is a regular
open cover of $X\times Y$ and $\mathcal{W}\prec \mathcal{U}\times \mathcal{V}
$.
\end{proof}

\begin{theorem}
\label{Th 4.2}Let $\left( X,f\right) $ and $\left( Y,h\right) $ be two
nearly compact $R$-dynamical systems. Then
\begin{equation*}
Ent_{n}\left( f\times h\right) \leq Ent_{n}\left( f\right) +Ent_{n}\left(
h\right) ,
\end{equation*}%
where $f\times h:X\times Y\rightarrow X\times Y$ is $R$-map defined by $%
\left( f\times h\right) \left( x,y\right) =\left( f\left( x\right) ,h\left(
y\right) \right) $.
\end{theorem}

\begin{proof}
Let $\mathcal{C}$ be an arbitrary regular open cover of $X\times Y$. By
Lemma \ref{lem 0.3}, there exist regular open cover $\mathcal{U}$ of $X$ and
regular open cover $\mathcal{V}$ of $Y$ such that $\mathcal{U}%
=\{U_{i}:i=1,\ldots ,n\}$ where $U_{i}$ is a regular open subset of $X$ and $%
\mathcal{V}=\{V_{j}:j=1,\ldots ,m\}$ where $V_{j}$ is a regular open subset
of $Y$ satisfying\ $\mathcal{C}\prec \mathcal{U}\times \mathcal{V}$. By
Theorem \ref{Th2.1} and Lemma \ref{cover}
\begin{eqnarray*}
Ent_{n}\left( f\times h,\mathcal{C}\right) &\leq &Ent_{n}\left( f\times h,%
\mathcal{U}\times \mathcal{V}\right) \\
&\leq &Ent_{n}\left( f,\mathcal{U}\right) +Ent_{n}\left( h,\mathcal{V}%
\right) .
\end{eqnarray*}%
So,
\begin{eqnarray*}
Ent_{n}\left( f\times h\right) &=&\sup_{\mathcal{C}}\{Ent_{n}\left\{ f\times
h,\mathcal{C}\right\} \\
&\leq &\sup_{\mathcal{C}}\left\{ Ent_{n}\left( f,\mathcal{U}\right)
+Ent_{n}\left( h,\mathcal{V}\right) \right\} \\
&\leq &\sup_{\mathcal{U}}\left\{ Ent_{n}\left( f,\mathcal{U}\right) \right\}
+\sup_{\mathcal{V}}\left\{ Ent_{n}\left( h,\mathcal{V}\right) \right\} \\
&=&Ent_{n}\left( f\right) +Ent_{n}\left( h\right) .
\end{eqnarray*}
\end{proof}

\begin{theorem}
Let $\left( X,f\right) $ and $\left( Y,h\right) $ be two topological $R$%
-dynamical systems with $X\times Y$ are $R$-space. If $X$ and $Y$ are
Hausdorff, then $Ent_{N}\left( f\times h\right) \leq Ent_{N}\left( f\right)
+Ent_{N}\left( h\right) $, where $f\times h:X\times Y\rightarrow X\times Y$
is $R$-map defined by $\left( f\times h\right) \left( x,y\right) =\left(
f\left( x\right) ,h\left( y\right) \right) $.
\end{theorem}

\begin{proof}
If $H\left( X\times Y,f\times h\right) =\emptyset $, then the result follows
since $Ent_{N}\left( f\times h\right) =0$ by Definition \ref{def_N2}.

If $H\left( X\times Y,f\times h\right) \neq \emptyset $, let $K\in H\left(
X\times Y,f\times h\right) $ and the projections $T_{x}:X\times Y\rightarrow
X$, $T_{y}:X\times Y\rightarrow Y$. For any regular open cover $\mathcal{W}$
of $X\times Y$, $T_{x}\left( K\right) \in H\left( X,f\right) ,T_{y}\left(
K\right) \in H\left( X,f\right) $ and $K\subseteq T_{x}\left( K\right)
\times T_{y}\left( K\right) $ by Lemma \ref{lem 0.1}. Since $X$ and $Y$ are
Hausdorff and $R$-spaces, then by Lemmas \ref{fA} and \ref{lem haus3} $%
T_{x}\left( K\right) $ and $T_{y}\left( K\right) $ are regular closed subset
of $X$ and $Y$, respectively.

Denote $T_{x}\left( K\right) $ and $T_{y}\left( K\right) $ as $K_{x}$ and $%
K_{y}$, respectively. Thus $K\subseteq K_{x}\times K_{y}$\ and $K_{x}\times
K_{y}\in H\left( X\times Y,f\times h\right) $ by Lemma \ref{lem 1}.\ By
Theorem \ref{Th 2.2},
\begin{equation*}
Ent_{N}\left( f\times h,\mathcal{W},K\right) \leq Ent_{N}\left( f\times h,%
\mathcal{W},K_{x}\times K_{y}\right) .
\end{equation*}%
From Theorem \ref{n=N}, we have%
\begin{equation}
Ent_{N}\left( f\times h,\mathcal{W},K_{x}\times K_{y}\right) =Ent_{n}\left(
\left( f\times h\right) |_{K_{x}\times K_{y}},\mathcal{W}|_{K_{x}\times
K_{y}}\right)  \label{(1)}
\end{equation}%
where $\mathcal{W}|_{K_{x}\times K_{y}}$ is a regular open cover of $%
K_{x}\times K_{y}$. By Lemma \ref{lem 0.3} , there exist a regular open
cover of $\mathcal{U^{\prime }}$ of $K_{x}$ and a regular open cover $%
\mathcal{V^{\prime }}$ of $K_{y}$ such that $\mathcal{U^{\prime }}=\left\{
U_{i}:i=1,\ldots ,n\right\} $ where $U_{i}$ is a regular open subset of $%
K_{x}$ and $\mathcal{V^{\prime }}=\left\{ V_{j}:j=1,\ldots ,n\right\} $
where $V_{j}$ is a regular open subset of $K_{y}$ satisfying $\mathcal{W}%
|_{K_{x}\times K_{y}}\prec \mathcal{U^{\prime }}\times \mathcal{V^{\prime }}$%
. For every $U_{i}$, there exist a regular open subset $A_{i}$ of $X$ such
that $U_{i}=A_{i}\cap K_{x}$. Denote $\mathcal{U}=\left\{ A_{1},\ldots
,A_{n},X\setminus K_{x}\right\} $. Hence $\mathcal{U}$ is regular open cover
of $X$ and $\mathcal{U}|_{K_{x}}=\mathcal{U^{\prime }}\cup \left\{ \emptyset
\right\} $. Similarly, there exist a regular open cover $\mathcal{V}$ of $Y$
such that $\mathcal{V}|_{K_{y}}=\mathcal{V^{\prime }}\cup \left\{ \emptyset
\right\} $. Hence, from $\mathcal{W}|_{K_{x}\times K_{y}}\prec \mathcal{%
U^{\prime }}\times \mathcal{V^{\prime }}$, we have%
\begin{equation*}
\mathcal{W}|_{K_{x}\times K_{y}}\prec \left( \mathcal{U^{\prime }}\cup
\left\{ \emptyset \right\} \right) \times \left( \mathcal{V^{\prime }}\cup
\left\{ \emptyset \right\} \right) .
\end{equation*}%
Therefore, we have $\mathcal{W}|_{K_{x}\times K_{y}}\prec \mathcal{U}%
|_{K_{x}}\times \mathcal{V}|_{K_{y}}$. Moreover, we also have%
\begin{equation}
(\mathcal{U}\times \mathcal{V})|_{K_{x}\times K_{y}}=\mathcal{U}%
|_{K_{x}}\times \mathcal{V}|_{K_{y}}.  \label{eq1}
\end{equation}%
By Theorem \ref{thm_n2.1}, equation (\ref{eq1}) and Theorem \ref{Th 4.2}, we
have%
\begin{eqnarray*}
Ent_{n}\left( \left( f\times h\right) |_{K_{x}\times K_{y}},\mathcal{W}%
|_{K_{x}\times K_{y}}\right) &\leq &Ent_{n}\left( \left( f\times h\right)
|_{K_{x}\times K_{y}},\mathcal{U}|_{K_{x}}\times \mathcal{V}|_{K_{y}}\right)
\\
&=&Ent_{n}\left( \left( f\times h\right) |_{K_{x}\times K_{y}},(\mathcal{U}%
\times \mathcal{V})|_{K_{x}\times K_{y}}\right) \\
&=&Ent_{n}\left( f|_{K_{x}},\mathcal{U}|_{K_{x}}\right) +Ent_{n}\left(
h|_{K_{y}},\mathcal{V}|_{K_{y}}\right) .
\end{eqnarray*}%
By Theorem \ref{n=N},%
\begin{equation*}
Ent_{N}\left( f\times h,\mathcal{U}\times \mathcal{V},K_{x}\times
K_{y}\right) =Ent_{n}\left( \left( f\times h\right) |_{K_{x}\times
K_{y}},\left( \mathcal{U}\times \mathcal{V}\right) |_{K_{x}\times
K_{y}}\right) .
\end{equation*}%
Hence, $Ent_{N}\left( f,\mathcal{U},K_{x}\right) =Ent_{n}\left( f|_{K_{x}},%
\mathcal{U}|_{K_{x}}\right) $ and $Ent_{N}\left( h,\mathcal{V},K_{y}\right)
=Ent_{n}\left( h|_{K_{y}},\mathcal{V}|_{K_{y}}\right) $. Thus from equation $%
\left( \ref{(1)}\right) $,%
\begin{eqnarray*}
Ent_{N}\left( f\times h,\mathcal{U}\times \mathcal{V},K_{x}\times
K_{y}\right) &\leq &Ent_{n}\left( f|_{K_{x}},\mathcal{U}|_{K_{x}}\right)
+Ent_{n}\left( h|_{K_{y}},\mathcal{V}|_{K_{y}}\right) \\
&=&Ent_{N}\left( f,\mathcal{U},K_{x}\right) +Ent_{N}\left( h,\mathcal{V}%
,K_{y}\right) .
\end{eqnarray*}%
By Theorem \ref{Th 2.2} and equation $\left( \ref{(1)}\right) $,%
\begin{eqnarray*}
Ent_{N}\left( f\times h,\mathcal{W},K\right) &\leq &Ent_{N}\left( f\times h,%
\mathcal{W},K_{x}\times K_{y}\right) \\
&=&Ent_{n}\left( \left( f\times h\right) |_{K_{x}\times K_{y}},\mathcal{W}%
|_{K_{x}\times K_{y}},K_{x}\times K_{y}\right) \\
&\leq &Ent_{n}\left( f|_{K_{x}},\mathcal{U}|_{K_{x}},K_{x}\right)
+Ent_{n}\left( h|_{K_{y}},\mathcal{V}|_{K_{y}},K_{y}\right) \\
&=&Ent_{N}\left( f,\mathcal{U},K_{x}\right) +Ent_{N}\left( h,\mathcal{V}%
,K_{y}\right) \\
&\leq &Ent_{N}\left( f\right) +Ent_{N}\left( h\right) .
\end{eqnarray*}%
Finally by Definitions \ref{def_N1} and \ref{def_N2}, we have%
\begin{eqnarray*}
Ent_{N}\left( f\times h\right) &=&\sup_{K}\left\{ \sup_{\mathcal{W}}\left\{
Ent_{N}\left( f\times h,\mathcal{W},K\right) \right\} \right\} \\
&\leq &Ent_{N}\left( f\right) +Ent_{N}\left( h\right) .
\end{eqnarray*}
\end{proof}

In the next example, we shows $Ent_{N}\left( f\right) =0$ for $f\left(
x\right) =kx$ on $%
\mathbb{R}
$ with usual topology.

\begin{example}
Let a map $f:\mathbb{R}\rightarrow \mathbb{R}$ define by $f\left( x\right)
=kx$, where $k\in \left\{ 2,3,4,\ldots \right\} $. The only invariant
compact subset of $%
\mathbb{R}
$ and hence nearly compact subset of $%
\mathbb{R}
$ is $\left\{ 0\right\} $, so $H\left( \mathbb{R},f\right) =\left\{ \left\{
0\right\} \right\} $. Denote $K=\left\{ 0\right\} $, thus for any regular
open cover $\mathcal{U}$ of $\mathbb{R}$, any subcover $\mathcal{V}$ of $%
\mathcal{U}$ such that $K\subseteq \dbigcup\limits_{V\in \mathcal{V}}V$ with
the smallest cardinality contains a single element of $\mathcal{U}$, i.e., $%
\left\{ 0\right\} $. Hence, $N_{K}\left( \mathcal{U}\right) =1$, so that $%
N_{K}\left( \bigvee_{i=0}^{n-1}f^{-i}\left( \mathcal{U}\right) \right) =1$
and $M_{K}\left( \bigvee_{i=0}^{n-1}f^{-i}\left( \mathcal{U}\right) \right)
=0$ which implies%
\begin{equation*}
Ent_{N}\left( f,\mathcal{U},K\right) =\lim_{n\rightarrow \infty }\frac{1}{n}%
M_{K}\left( \bigvee_{i=0}^{n-1}f^{-i}\left( \mathcal{U}\right) \right) =0.
\end{equation*}%
Therefore, from Definition \ref{def_N2}, $Ent_{N}\left( f\right) =0$. If we
replaced $\mathbb{R}$ by $\mathbb{R}^{+}$, then $H\left( \mathbb{R}%
^{+},f\right) =\emptyset $. Again from Definition \ref{def_N2}, $%
Ent_{N}\left( f\right) =0$. Recall that the role of topological entropy is
to describe the complexity of a topological dynamical system and positivity
of entropy implies the complexity of dynamical behaviour of the system.
Thus, we could say that topological nearly entropy is describe the
complexity of topological $R$-dynamical system and it positive value shows
that more complicated $R$-dynamical properties of the system. Hence, we
conclude that $Ent_{N}\left( f\right) =0$ is a characterization for $R$%
-dynamical property of the map $f\left( x\right) =kx$.
\end{example}

\section{Conclusion}

We present a notion of topological nearly entropy on nearly compact spaces
and investigate some its properties which is an extension from our previous
paper \cite{gulam}. Both of the notions are coincide in the special case of
nearly compact spaces. Yet, they retain various fundamental properties of
Adler et al. \cite{adler}. A new space, namely, $R$-space was introduced for
nearly compact on Hausdorff space. Consequently, the properties of
topological nearly entropy for product spaces are obtained.


\begin{thebibliography}{99}
\bibitem{adler} R. L. Adler, A. G. Konheim and M. H. McAndrew, Topological
entropy, Trans. Amer. Math. Soc. 144 (1965), 309-319.

\bibitem{uzzal} B. M. U. Afzan, Topological $H$-entropy, J. Adv. Res. Dyn.
Control Syst. 7 (2015), 96-104.

\bibitem{bowen} R. Bowen, Entropy for group endomorphisms and homogeneous
spaces, Trans. Amer. Math. Soc. 153 (1971), 401-414.

\bibitem{canovas} J. S. Canovas and J. M. Rodriguez, Topological entropy of
maps on the real line, Top. Appl. 153 (2005), 735-746.

\bibitem{carnahan} D. Carnahan, Some properties related to topological
space, Ph.D. Thesis, Univ. of Arkansas, 1973.

\bibitem{ekici} E. Ekici, On $\delta $-semiopen sets and a generalization of
functions, Bol. Soc. Paran. Mat. 23 (2005), 73-84.

\bibitem{kilicman} A. J. Fawakhreh and A. K{\i }l{\i }\c{c}man, Mappings and
some decompositions of continuity on nearly Lindel\"{o}f space, Acta Math.
Hungar. 97 (2002), 199-206.

\bibitem{goodwyn} L. W. Googwyn, Topological entropy bounds measure-theoretc
entropy, Proc. Amer. Math. Soc. 23 (1969), 679-688.

\bibitem{gulam} S. Gulamsarwar and Z. Salleh, Topoogical nearly entropy, AIP
Conference Proceedings, 1870, 030006 (2017); doi 10.1063/1.4995831

\bibitem{liu} L. Liu, Y. Wang and G. Wei, Topological entropy of continuous
functions on topological spaces, Chaos Soliton Fractals 39 (2007), 417-427.

\bibitem{zabidin} Z. Salleh, A note on topological entropy of continuous
self-maps, J. Math. Sys. Sc. 5 (2015), 93-99.

\bibitem{singal} M. K. Singal and A. Mathur, On nearly-compact space, Boll.
Un. Mat. 4 (1969), 702-710.

\bibitem{walter} P. Walters, An Intoduction to Ergodic Theory, Springer, New
York, 1982.
\end{thebibliography}
\end{document}